\documentclass{amsart}
\usepackage{graphicx}
\newtheorem{thm}{Theorem}[section]
\newtheorem{cor}[thm]{Corollary}
\newtheorem{lem}[thm]{Lemma}

\theoremstyle{definition}

\theoremstyle{remark}
\newtheorem{rem}[thm]{Remark}
\numberwithin{equation}{section}

\newcommand{\R}{\mathbb R}

\newcommand{\cP}{\mathcal{P}}
\newcommand{\M}{\mathcal{M}}
\newcommand{\B}{\mathcal{B}}

\begin{document}

\title[Some estimates for imaginary powers of the Laplace operator ]{Some estimates for imaginary powers of the  Laplace operator in variable Lebesgue spaces and applications}%

\author{Alberto Fiorenza}
\address{Alberto Fiorenza \\
Dipartimento di Architettura \\
Universit\`a di Napoli \\ Via Monteoliveto, 3 \\
I-80134 Napoli, Italy\\
and Istituto per le Applicazioni del Calcolo
``Mauro Picone", sezione di Napoli \\
Consiglio Nazionale delle Ricerche \\
via Pietro Castellino, 111 \\
 I-80131 Napoli, Italy }
\email{fiorenza@unina.it}

\author{Amiran Gogatishvili}
\address{Amiran Gogatishvili \\
Institute of Mathematics of the \\
Academy of Sciences of the Czech Republic \\
\'Zitna 25 \\
115 67 Prague 1, Czech Republic}
\email{gogatish@math.cas.cz}

\author{Tengiz Kopaliani}
\address{Tengiz Kopaliani \\
Faculty of Exact and Natural Sciences\\
I. Javakhishvili Tbilisi State University\\
 University St. 2\\
 0143 Tbilisi, Georgia}
\email{tengiz.kopaliani@tsu.ge}


\thanks{The research of the second and third authors was partly supported by Shota Rustaveli National Science Foundation
grants no.13/06 (Geometry of function spaces, interpolation and embedding theorems) and 31/48 (Operators in some function spaces and their applications in Fourier Analysis). The research of the  second author was  partly supported by grant no. P201-13-14743S  of the Grant Agency of the Czech Republic and RVO: 67985840.}

\begin{abstract}

In this paper we study some estimates of norms in variable exponent Lebesgue spaces for singular integral operators that are imaginary powers of the Laplace operator in $\R^n$.  Using the Mellin transform argument, from these estimates we obtain the boundedness for a family of maximal operators in variable exponent Lebesgue spaces, which are closely related to the (weak) solution of the wave equation.
\end{abstract}
\maketitle
\section{Introduction}

In the recent paper \cite{ FGK} we studied the boundedness of  Stein's spherical maximal function $\mathcal M$ in variable exponent Lebesgue spaces. The proof is based on the Rubio De Francia extrapolation method and one of the corresponding results in weighted Lebesgue spaces.   The Stein's spherical maximal functions are closely related to the solution of the wave equation in  $\R^3$. In order to study the wave equation in  $\R^n$ , $n>3$ we need to consider the more general spherical maximal function $\mathcal M^\alpha$, $\alpha =\frac{3-n}{2}$ (\cite{S1}).
To investigate such operators we use a new approach based on a Mellin transform argument used for the first time by Cowling and  Mauceri in \cite{CM}, which reduces the problem to that one to find sharp estimates for norms of imaginary power of the Laplace operator.  
 
%


 Let $\cP(\R^n)$ be the set of all measurable functions $p: \mathbb{R}^{n} \to[1,\infty]$, which will be
called variable exponents. Set $p^-=\mbox{essinf}_{x\in \mathbb{R}^{n}}p(x)$ and $p^+=\mbox{esssup}_{x\in\mathbb{R}^{n}}p(x)$.
If $p^+ < \infty$, then the variable exponent $p$ is bounded. For $p\in \cP(\R^n)$, let $p'\in \cP(\R^n)$ be defined through $\frac{1}{p(y)} + \frac{1}{p'(y)} = 1$, where we adopt the convention $\frac1\infty:= 0$. The function $p'$ represents the dual variable exponent of $p$ (and the symbol should not be confused with the derivative).

For $p\in \cP(\R^n)$,  let $\mathbb{R}^{n}_\infty$ be the set where $p=+\infty$. The symbol $L^{p(\cdot)}(\mathbb{R}^{n})$ denotes the set of
measurable functions real or complex valued $f$ on $\mathbb{R}^{n}$ such that for some
$\lambda>0$
$$\rho_{p(\cdot)}\left(\frac{f}{\lambda}\right)=
\int_{\mathbb{R}^{n}\backslash \mathbb{R}^{n}_\infty}\left(\frac{|f(x)|}{\lambda}
\right)^{p(x)}dx+\left\|\frac{f}{\lambda}\right\|_{L^\infty(\mathbb{R}^{n}_\infty)}<\infty.
$$
This set becomes a Banach function space when equipped with the norm
$$
\|f\|_{p(\cdot)}=\inf\left\{\lambda>0:\,\,\rho_{p(\cdot)}\left(\frac{f}{\lambda}\right)\leq 1 \right\}.
$$

Let $B(x,r)$ denote the open ball in $\mathbb{R}^{n}$ of radius $r$
and center $x.$ By $|B(x,r)|$ we denote the $n-$dimensional Lebesgue
measure of $B(x,r).$  The Hardy-Littlewood maximal operator $M$ is
defined on locally integrable functions $f$ on $\mathbb{R}^{n}$ by
the formula
$$
Mf(x)=\sup_{r>0}\frac{1}{|B(x,r)|}\int_{B(x,r)}|f(y)|dy.
$$

Define the Spherical Maximal operator $\mathcal{M}$ by
$$
\mathcal{M}f(x):=\sup_{t>0}\left\vert\mu_{t}*f(x)\right\vert=
\sup_{t>0}\left\vert\int_{\{y\in \mathbb{R}^{n}:|y|=1\}} \,
f(x-ty)d\mu_{t}(y)\right\vert
$$
 where  $\mu_{t}$ denotes the
normalized surface measure on the sphere of center $0$ and radius
$t$ in $\mathbb{R}^{n}$. The Hardy-Littlewood maximal operator $M$,
which involves averaging over balls, is clearly related to the
spherical maximal operator, which averages over spheres. Indeed, by
using polar coordinates, one easily verifies the pointwise
inequality $Mf(x)\leq \mathcal{M} f(x)$ for any (continuous)
function.


A function $p: \mathbb{R}^{n} \to(0,\infty)$ is said to be locally log-H\"older continuous on $ \mathbb{R}^{n} $ if there exists $c_1 >
0$ such that
$$|p(x)-p(y)|\le c_1 \frac{1}{\log(e + 1/|x- y|)}$$
for all $x, y\in  \mathbb{R}^{n} $, $|x- y|<1/2$.  Moreover, $p(\cdot)$
satisfies the log-H\"older decay condition if there exist
$p_\infty\in (0,\infty)$ and a constant $c_2 > 0$ such that
$$|p(x)-p_\infty|\le c_2 \frac{1}{\log(e + |x|)}$$
for all $x\in  \mathbb{R}^{n} $. We say that $p(\cdot)$ is globally
log-H\"older continuous on $ \mathbb{R}^{n} $  if it is locally log-H\"older continuous and satisfies
the log-H\"older decay condition.
We write $ p(\cdot)\in
\cP^{\log}(\R^n)$ if $p(\cdot)\in \cP(\R^n)$ and $\frac{1}{p(\cdot)}$  is globally
log-H\"older continuous on $ \mathbb{R}^{n}$
If $p(\cdot)\in \cP(\R^n)$  with $p_+ <\infty$, then $ p(\cdot)\in
\cP^{\log}(\R^n)$ if and only if 
$p$ is globally log-H\"older continuous on $\R^n$.
If $p\in \cP^{\log}(\R^n)$ and  $p^->1$, then the classical boundedness theorem
for the Hardy-Littlewood maximal operator can be extended to
$L^{p(\cdot)}$ (see \cite{CUFN1, CUFN2, CDF}). For more information on variable Lebesgue spaces, log-H\"older continuity conditions and their relationship with the boundedness of the Hardy-Littlewood maximal operator, see the monographs \cite{CUF, DHHR}. If $n\ge 3$, $p\in \cP^{\log}(\R^n)$ and $\frac{n}{n-1}<p^-\le p^+<p^-(n-1)$, then the boundedness theorem for the spherical maximal function $\M$  in $L^{p(\cdot)}$  was proved in \cite{FGK}. 

We will denote by $\B(\R^n)$ the class all measurable functions $p: \mathbb{R}^{n} \to(0,\infty)$  for which  the Hardy-Littlewood maximal operator is bounded on $L^{p(\cdot)}$.

Throughout the paper, we denote by $c$, $C$, $c_1$, $C_1$, $c_2$, $C_2$,  etc. 
positive constant which is independent of the main parameters but which may vary
from line to line.

\section{Imaginary power of Laplace operator in variable Lebesgue spaces }

Let $S(\mathbb{R}^{n})$ denote the Schwartz space, consisting of
sufficiently smooth functions that are rapidly decreasing at
infinity. 
Let $\Delta$ be the standard Laplace operator in $\R^n$, given by
$$\Delta = \sum_{j=1}^n\partial_j^2.$$
If $\hat{f}(\xi)=\int_{\R^n} e^{2\pi x\cdot\xi} f(x)dx$, then 
$$
(-\triangle
f)^{\wedge}(\xi)=|2\pi\xi|^{2}\widehat{f}(\xi),\,\,\,\,f\in
S(\mathbb{R}^{n}).$$

Starting from this relation, it is natural to define $\Delta^{\beta/2}$ for any complex exponent
$\beta$ by
$$
((-\Delta)^{\beta/2})^{\wedge}(\xi)=(2\pi|\xi|)^{\beta}\widehat{f}(\xi),\,\,\,\,f\in
S(\mathbb{R}^{n}).
$$

 In particular, for each $0 <\alpha<n$, the operator 
 $$I_\alpha: f\mapsto(-\Delta)^{-\alpha/2}f
 $$
is known as the Riesz potential. Here $I_\alpha$ may be expressed as
$$I_\alpha f=K_\alpha* f$$
where $K_\alpha(x) =\pi^{-n/2}2^{-\alpha}\Gamma\left(\frac{n-\alpha}{2}\right)/\Gamma\left(\frac{\alpha}{2}\right)|x|^{-n+\alpha}$, the symbol $\Gamma$ denoting the standard gamma function, and therefore (see \cite{S2} p. 117; see also \cite{CUMP, CUF}) $I_\alpha$ is an integral
operator.

In this paper we shall consider the operator
$I_{iu},\,\,u\in\,\mathbb{R}\backslash\{0\},$ given by
$$
I_{iu}f=K_{iu}\ast f,\,\,\,f\in S(\mathbb{R}^{n}),
$$
which makes sense via
$$
(I_{iu}f)^{\wedge}(\xi)=(2\pi|\xi|)^{-iu}\widehat{f}(\xi),\,\,\,f\in
S(\mathbb{R}^{n}),
$$
that is $I_{iu}=(-\Delta)^{\frac{-iu}{2}}$, an imaginary power of
$-\Delta.$ This operator was studied by  Muckenhhoupt \cite{M} in 1960 and
used by Cowling and Mauceri \cite{CM} in 1978 to prove E.M.~Stein's theorem on the
spherical maximal function \cite{S1}. 

Note that $|\widehat{K}(\xi)|=|(2\pi|\xi|)^{-iu}|=1,$ so that by
Plancherel's theorem we have in $L^{2}(\mathbb{R}^{n})$
\begin{equation} \label{2.1}
\|I_{iu}f\|_{2}=\|f\|_{2}.
\end{equation}

By   using further properties of the kernel $K_{iu}$, particularly
the fact that it is locally integrable away from the origin and
satisfies
$$
|K_{iu}(x)|\leq C(1+|u|)^{n/2}|x|^{-n}
$$
and
$$
|\nabla K_{iu}(x)|\leq C(1+|u|)^{n/2+1}|x|^{-n-1}
$$
for $x\neq0$ (see \cite{CM} and \cite{G}), one may observe that $I_{iu}$ also extends to a
bounded operator on $L^{p}_{w}(\mathbb{R}^{n})$.
By $w$ we mean a weight, i.e. a non-negative, locally  integrable function on $\mathbb{R}^{n}$. When $1<p<\infty,$ we say $w\in A_{p}$ if for every ball   $B$
$$
\frac{1}{|B|}\int_{B}w(x)dx\left(\frac{1}{|B|}\int_{B}w(x)^{1-p'}dx\right)^{p-1}
\leq C<\infty.
$$

By $A_{p,w}$ we denote the infimum over the constants on the
right-hand side of the last inequality.


\begin{thm}[\cite{G}] \label{th2.1}
Let be $1<p<\infty$ and $w\in A_{p}.$ For each
$\delta\in (0,1)$ and $u\in\mathbb{R}\backslash\{0\},$
the following
weighted estimate holds wenhever $w\in A_{p},\,1<p<\infty:$
\begin{equation} \label{2.2}
\|I_{u}f\|_{p,w}\leq C(1+|u|)^{n/2+\delta}\|f\|_{p,w}, \,\,\,f\in
L^{p}_{w}(\mathbb{R}^{n}).
\end{equation}
\end{thm}

Our extension of this theorem in the variable Lebesgue spaces setting is the following 

\begin{thm}\label{th2.2}
Let be $p(\cdot)\in\B(\R^n)$.  Then for all $\delta\in (0,1)$ there exists a constant $C$ such that, for all
$u\in\mathbb{R}\backslash\{0\}$
\begin{equation} \label{2.3}
\|I_{iu}f\|_{p(\cdot)}\leq C(1+|u|)^{n/2+\delta}\|f\|_{p(\cdot)},
\,\,\,f\in L^{p(\cdot)}(\mathbb{R}^{n}).
\end{equation}
\end{thm}

To  prove this Theorem we need the extrapolation theorem for variable Lebesgue spaces. 
By $\mathcal{F}$  we will denote a family of ordered pairs of
non-negative, measurable functions $(f,g).$ We say that an
inequality
\begin{equation} \label{2.4}
\int_{\mathbb{R}^{n}}f(x)^{p_{0}}w(x)dx\leq C
\int_{\mathbb{R}^{n}}g(x)^{p_{0}}w(x)dx,\,\,\,(0<p_{0}<\infty)
\end{equation}
holds  for any $(f,g)\in\mathcal{F}$ and $w\in A_{q}$ (for some
$q,\,1<q<\infty$) if it holds for any pair in $\mathcal{F}$ such
that the left-hand side is finite, and the constant $C$ depends only
on $p_{0}$ and on the constant $A_{q,w}.$

\begin {thm}(\cite[Theorem~7.2.1, page 214]{DHHR}, see also \cite[Theorem~4.25, page 87]{CUMP}, \cite[Chap.~5]{CUF}). \label{th2.3}
Given a family $\mathcal{F}$, assume that \eqref{2.4} holds for some
$1<p_{0}<\infty,$ for every weight $w\in A_{p_{0}}$ and
 for all $(f,g)\in \mathcal{F}$. Let exponent
$p(\cdot)$ be such that there
 exists $1<p_{1}<p_{-},$ with $(p(\cdot)/p_{1})'\in \mathcal{B}(\mathbb{R}^{n}).$
Then
$$
\|f\|_{p(\cdot)}\leq C\|g\|_{p(\cdot)}
$$
for all $(f,g)\in\mathcal{F}$ such that $f\in
L^{p(\cdot)}(\mathbb{R}^{n}).$
\end{thm}

{\it Proof of the Theorem~\ref{th2.2}.}
Using Theorem~\ref{th2.3},  estimate \eqref{2.2} and the fact that if
$p(\cdot)\in\mathcal{B}(\mathbb{R}^{n})$ then $(p(\cdot)/p_{1})'\in
\mathcal{B}(\mathbb{R}^{n})$ for some $1<p_{1}<p_{-},$ (see \cite[Theorem~5.7.2, page 181]{DHHR}) we
obtain \eqref{2.3}.
$\square$

\begin{cor} \label{c2.4}
 Let $\frac{1}{p(\cdot)}=\frac{1-\theta}{2}+\frac{\theta}{\widetilde p(\cdot)}$ for some $0<\theta<1$ and
$\widetilde{p}(\cdot)\in\mathcal{B}(\mathbb{R}^{n})$.  
Then for all $0<\delta<1$ there exists a constant $C$
such that, for all $u\in\mathbb{R}\backslash\{0\}$
\begin{equation} \label{2.5}
\|I_{iu}f\|_{p(\cdot)}\leq C(1+|u|)^{\theta
n/2+\theta\delta}\|f\|_{p(\cdot)}, \,\,\,f\in
L^{p(\cdot)}(\mathbb{R}^{n}).
\end{equation}
\end{cor}
\begin{proof} By using the complex interpolation theorem for variable exponent  Lebesgue spaces (see \cite[Theorem.1.2, page 215]{DHHR} ), we have
$L^{p(\cdot)}(\mathbb{R}^{n}=[L^{2}(\mathbb{R}^{n},\,\,L^{\widetilde{p}(\cdot)}(\mathbb{R}^{n}]_{\theta}$.
Therefore, $p(\cdot)\in \B(\R^n)$ and  
$$
\|I_{iu}\|_{L^{p(\cdot)}\rightarrow
L^{p(\cdot)}}\leq\|I_{u}\|_{L^{2}\rightarrow
L^{2}}^{1-\theta}\|I_{u}\|_{L^{\widetilde{p}(\cdot)}\rightarrow
L^{\widetilde{p}(\cdot)}}\leq C(1+|u|)^{\theta n/2+\theta\delta}.
$$
\end{proof}

\section{The spherical maximal function}

For $\alpha>0,$  let
$m_{\alpha}(x)=(1-|x|^{2})^{\alpha-1}/\Gamma(\alpha),$ where
$|x|<1,$ and $m_{\alpha}(x)=0$ if $|x|\geq1.$  With
$m_{\alpha,t}(x)=m_{\alpha}(x/t)t^{-n},\,\,t>0,$  we define
spherical means of (complex) order ${\mathrm Re}\,\alpha>0,$  by
$$
\mathcal{M}_{t}^{\alpha}f(x)=(m_{\alpha,t}\ast f)(x).
$$

Note that the Fourier transform of $m_{\alpha}$ is given by
$$
\widehat{m}_{\alpha}(\xi)=\pi^{-\alpha+1}|\xi|^{-n/2-\alpha+1}J_{n/2+\alpha-1}(2\pi|\xi|).
$$

 The  definition  of $\mathcal{M}_{t}^{\alpha}$   can be extended to the region
${\mathrm Re}\,\alpha\leq0$  by the analytic continuation. Indeed  for complex
$\alpha$ in general we can  define the operator
$\mathcal{M}_{t}^{\alpha}$ by
$$
(\mathcal{M}_{t}^{\alpha}f)^{\wedge}(\xi)=\widehat{m}_{\alpha}(t\xi)\widehat{f}(\xi),
\,\,\,f\in S(\mathbb{R}^{n}).
$$

The spherical maximal operator of order $\alpha$ is defined by
\begin{equation} \label{sphmaxop}
\mathcal{M}^{\alpha}f(x)=\sup_{t>0}|M_{t}^{\alpha}f(x)|.
\end{equation}

Note that for  $\alpha=0$ we have
$\mathcal{M}^{\alpha}f(x)=c\mathcal{M}f(x)$ for an appropriate constant $c.$

\begin{thm}[E.M.~Stein \cite{S1}]
The inequality $\|\mathcal{M}^{\alpha}f\|_{p}\leq
A_{p,\alpha}\|f\|_{p}$ holds in the following circumstances:

(a)  if $1<p\leq2,$ when $\alpha>1-n+n/p,$

(b) if  $2\leq p\leq\infty,$ when $\alpha>(1/p)(2-n).$
\end{thm}

Note that for $\alpha$ in Steins's theorem we have the restriction
$1-n/2<\alpha<1$ and, moreover, we have

 a) if  $\alpha=0,$  then $n\geq 3,\,\,p>\frac{n}{n-1},$

 b) If   $0<\alpha<1$  then $n\geq2,\,\frac{n}{n-1+\alpha}<p\leq\infty,$

 c) If   $1-n/2<\alpha<0$ then $n\geq3,\,\,\frac{n}{n-1+\alpha}<p<\frac{2-n}{\alpha}.$

 In this section we study boundedness properties of the
 Stein's spherical maximal operator  $\M^\alpha$ on the variable exponent Lebesgue
spaces.  Our main result is following

\begin{thm} \label{th3.2}
Let $1-n/2<\alpha <1$ and let $\widetilde{p}(\cdot):=\frac{2p(\cdot)\theta}{2-(1-\theta)p(\cdot)} \in\mathcal{B}(\mathbb{R}^{n})$ for some 
$0<\theta <1-\frac{2}{n}+\frac{2}{n}\alpha$.
Then the spherical maximal operator $\mathcal{M}^{\alpha}$ is bounded on
$L^{p(\cdot)}(\mathbb{R}^{n}).$
\end{thm}
\begin{proof}
Let
$F_{\alpha}(\lambda)=\pi^{-\alpha+1}\lambda^{-n/2-\alpha+1}J_{n/2+\alpha-1}(2\pi\lambda),\,\,\,\lambda>0,$
where $J_{\nu}$ denotes the Bessel function of order $\nu$.

Then we have
$$
(M_{t}^{\alpha}f)^{\wedge}(\xi)=F_{\alpha}(t|\xi|)\widehat{f}(\xi)
$$

Let
$F^{\ast}_{\alpha}(\lambda)=F_{\alpha}(\lambda)-F_{\alpha}(0)e^{-\lambda^{2}}$,
so that $F^{\ast}_{\alpha}(0)=0.$ Using the Mellin transform we have
$$
F^{\star}_{\alpha}(\lambda)=\int_{\mathbb{R}}A_{\alpha}(u)\lambda^{iu}du,\,\,\,\lambda>0,
$$
that is, $F^{\ast}_{\alpha}$ is the Mellin transform  of $A_\alpha(u)$ (for the Mellin transform see \cite{F}). By the Fourier Inversion Theorem, this holds if and only if
\[
A_{\alpha}(u)=\frac{1}{2\pi}\int_{0}^{\infty}F^{\star}_{\alpha}(\lambda)\lambda^{-1-iu}d\lambda,\,\,u\in\mathbb{R}.
\]
 For $f\in \mathcal{S}(R^{n})$ we have
\begin{align*}
m_{t}^{\alpha}\ast
f(x)&=(\widehat{m_{t}^{\alpha}}(\xi)\cdot\widehat{f}(\xi))^{\vee}(x)=(F_{\alpha}(t|\xi|)\cdot\widehat{f}(\xi))^{\vee}(x)\\
&=\left(\left(F_{\alpha}^{\ast}(t|\xi|)+F_{\alpha}(0)e^{-|t\xi|^{2}}\right)\cdot\widehat{f}(\xi)\right)^{\vee}(x).
\end{align*}

and
\begin{align*}
(F_{\alpha}^{\ast}(t|\xi|)\cdot\widehat{f}(\xi))^{\vee}(x)&=
\left(\int_{R}A_{\alpha}(u)(t|\xi|)^{iu}du
\widehat{f}(\xi)\right)^{\vee}(x)\\
&=\int_{R}A_{\alpha}(u)t^{iu}\left(|\xi|^{iu}
\widehat{f}(\xi)\right)^{\vee}(x)du
\\
&=\int_{R}A_{\alpha}(u)t^{iu}(2\pi)^{-iu}\left((2\pi |\xi|)^{iu}
\widehat{f}(\xi)\right)^{\vee}(x)du
\\
&=\int_{R}A_{\alpha}(u)t^{iu}(2\pi)^{-iu}I_{iu}f(x)du.
\end{align*}

Since $|(2\pi)^{-iu}t^{iu}| = 1$, by using Minkovski's inequality and Corollary~\ref{c2.4}, we obtain

\begin{align*}
\|(F_{\alpha}^{\ast}(t|\xi|)\cdot\widehat{f}(\xi))^{\vee}(x)\|_{p(\cdot)}&=\left\|\int_{R}A_{\alpha}(u)(2\pi)^{-iu}t^{iu}I_{iu}f(x)du\right \|_{p(\cdot)}\\
&\leq\int_{R}\|A_{\alpha}(u)(2\pi)^{-2i}t^{iu}I_{iu}f(x)\|_{p(\cdot)}du\\
&\le \int_{R}|A_{\alpha}(u)|\|I_{iu}f(\cdot)\|_{p(\cdot)}du\\
&\leq C\int_{R}|A_{\alpha}(u)|(1+|u|)^{\theta n/2+\theta \delta}du\|f\|_{p(\cdot)}.
\end{align*}

Using  the following expression for $A_{\alpha}(u)$ (see \cite{GB}) 
$$
A_{\alpha}(u)=\frac{\Gamma(\alpha+n/2-1/2)\Gamma(-\frac{iu}{2})}{4\pi^{1/2}}\left[\frac{2^{-iu}}{\Gamma(\alpha+n/2+iu/2)}-\frac{1}{\Gamma(\alpha+n/2)}\right]
$$
we have $A_{\alpha}(u)=O\left((1+|u|)^{-\alpha-n/2}\right).$

Finally we get 
\begin{align*}
\|(F_{\alpha}^{\ast}(t|\xi|)\cdot\widehat{f}(\xi))^{\vee}(x)\|_{p(\cdot)}
 &\leq C\int_{R}(1+|u|)^{-\alpha-n/2} (1+|u|)^{\theta n/2+\theta \delta}du\|f\|_{p(\cdot)}\\
&= C\int_{R}(1+|u|)^{-\alpha-n/2+\theta n/2+\theta \delta}du\|f\|_{p(\cdot)}
\end{align*} 
If we take $\delta$ very small, such that $  -\alpha-n/2+\theta n/2+\theta \delta<-1$, we  obtain that 
$$\int_{R}(1+|u|)^{-\alpha-n/2+\theta n/2+\theta \delta}du<\infty$$
and therefore 
\begin{equation} \label{3.1}
\|(F_{\alpha}^{\ast}(t|\xi|)\cdot\widehat{f}(\xi))^{\vee}(x)\|_{p(\cdot)}\le C\|f\|_{p(\cdot)}.
\end{equation}
It is known that (see \cite[p. 61]{S2}).
$$\left|\left(F_{\alpha}(0)e^{-|t\xi|^{2}}\cdot\widehat{f}(\xi)\right)^{\vee}(x)\right|\le C(P_t\ast |f|)(x),$$
where $P_t$ is the Poisson kernel.
Since (see \cite[Theorem 1, p.62]{S2})
$$\sup_{t>0}P_t\ast |f|(x) \le c Mf(x),$$
 where $M$  is the Hardy-Littlewood maximal function, using the assumption $\widetilde{p}(\cdot) \in\mathcal{B}(\mathbb{R}^{n})$, the identity $\frac{1}{p(\cdot)}=\frac{1-\theta}{2}+\frac{\theta}{\widetilde{p}(\cdot)}$ and the complex interpolation method  
 $$ L^{p(\cdot)}(\R^n)=[L^2(\R^n),L^{\widetilde{p}(\cdot)}(\R^n)]_{\theta}$$,
 we obtain that $p(\cdot)\in \B(\R^n)$, and therefore 
\begin{equation} \label{3.2}
\left\|\left(F_{\alpha}(0)e^{-|t\xi|^{2}}\cdot\widehat{f}(\xi)\right)^{\vee}(x)\right\|_{p(\cdot)}
\le C\|Mf\|_{p(\cdot)}\le C\|f\|_{p(\cdot)}.
\end{equation}

Using \eqref{3.1} and \eqref{3.2} we get 
$$\|\mathcal{M}^{\alpha}f\|_{p(\cdot)}\le C\|f\|_{p(\cdot)}.$$
\end{proof}

As a consequence, we can get another boundedness result for $\mathcal{M}^{\alpha}$ in the variable Lebesgue spaces. We need the following  

\begin{lem} \label{l3.3} Let $p\in\cP^{\log}(\R^n)$. If   $1-n/2<\alpha<1$ and   $\frac{n}{n-1+\alpha}<p_{-}\leq
p_{+}<\frac{n}{1-\alpha}$, then there are  exist  $\theta$,    $0<\theta<1-\frac{2}{n}+\frac{2}{n}\alpha$ and variable exponent 
 $\widetilde{p}(\cdot)\in\cP^{\log}(\mathbb{R}^{n})$, such that we have  
\begin{equation} \label{l1}
\frac{1}{p(\cdot)}=\frac{1-\theta}{2}+\frac{\theta}{\widetilde{p}(\cdot)}.
\end{equation}
\end{lem}
\begin{proof} We need to find $\theta$ such that $0<\theta<1-\frac{2}{n}+\frac{2}{n}\alpha$ and exponent
 $\widetilde{p}(\cdot)\in\cP^{\log}(\R^n) $  such that \eqref{l1} holds.

 We have
$$
\frac{1}{n}-\frac{\alpha}{n}<\inf_{x\in
\mathbb{R}^{n}}\frac{1}{p(x)}\leq \sup_{x\in
\mathbb{R}^{n}}\frac{1}{p(x)}<1-\frac{1}{n}+\frac{\alpha}{n}.
$$

Let $\frac{1}{p(x)}=\frac{1}{2}+r(x).$  It is easy to see that
\begin{equation} \label{l2}
\frac{1}{n}-\frac{\alpha}{n}-\frac{1}{2}<\inf_{x\in\mathbb{R}^{n}}r(x)\leq\sup_{x\in\mathbb{R}^{n}}r(x)<\frac{1}{2}-\frac{1}{n}+\frac{\alpha}{n}.
\end{equation}

Equation \eqref{l1}  is equivalent to
\begin{equation} \label{l3}
\frac{1}{2}+\frac{r(x)}{\theta}=\frac{1}{\widetilde{p}(x)}.
\end{equation}

Using \eqref{l2} we may  take a small $\delta>0$  such that
$$
\frac{1}{n}-\frac{\alpha}{n}-\frac{1}{2}+\delta<\inf_{x\in\mathbb{R}^{n}}r(x)\leq\sup_{x\in\mathbb{R}^{n}}r(x)<\frac{1}{2}-\frac{1}{n}+\frac{\alpha}{n}-\delta.
$$
Then for
$\theta,\,\,0<\theta<1-\frac{2}{n}+\frac{2}{n}\alpha,\,\,\theta=1-\frac{2}{n}+\frac{2}{n}\alpha-\theta_{0},\,\,\,\theta_{0}>0$
we have
$$
\frac{\frac{1}{n}-\frac{\alpha}{n}-\frac{1}{2}+\delta}{1-\frac{2}{n}+\frac{2}{n}\alpha-\theta_{0}}<\inf_{x\in\mathbb{R}^{n}}\frac{r(x)}{\theta}
\leq\sup_{x\in\mathbb{R}^{n}}\frac{r(x)}{\theta}<\frac{\frac{1}{2}-\frac{1}{n}+\frac{\alpha}{n}-\delta}{1-\frac{2}{n}+\frac{2}{n}\alpha-\theta_{0}}
$$
and taking $\theta_{0}<2\delta$  we get
\begin{equation} \label{l4}
-\frac{1}{2}<\inf_{x\in\mathbb{R}^{n}}\frac{r(x)}{\theta}\leq
\sup_{x\in\mathbb{R}^{n}}\frac{r(x)}{\theta}<\frac{1}{2}.
\end{equation}

From \eqref{l3} and \eqref{l4} it follows 
 $$
0<\inf_{x\in\mathbb{R}^{n}}\frac{1}{\widetilde{p}(x)}\leq\sup_{x\in\mathbb{R}^{n}}\frac{1}{\widetilde{p}(x)}<1.
 $$
Finally, it is not hard to prove that $\widetilde{p}(\cdot)\in\cP^{\log}(\R^n) $.  Indeed we may use the simple
	equality
$$
\left|\frac{1}{p(x)}-\frac{1}{p(y)}\right|=\theta\left|\frac{1}{\widetilde{p}(x)}-\frac{1}{\widetilde{p}(y)}\right|.
$$
\end{proof}

\begin{rem} In fact in the  Lemma~\ref{l3.3} we use only the simple fact that if $p\in \cP^{\log}(\R^n)$ then  also $\widetilde{p}(\cdot)\in\cP^{\log}(\R^n) $. For applications we need to consider  such classes of exponents with this property, because they make the  Hardy-Littlewood maximal function bounded on variable exponent Lebesgue spaces. In general  the class of variable exponents $\B(\R^)$, for which the  Hardy-Littlewood maximal function are bounded on variable exponent Lebesgue spaces have not such property (see \cite{L} and the digression in \cite[Chap.~4]{CUF}). The opposite statement by interpolation theorem is always true, that is if  $\widetilde{p}(\cdot)\in\B(\R^n) $, then   $p\in \B(\R^n)$. 
For example (see \cite[Remark 4.3.10, page 114]{DHHR}), it is possible to replace the log-H\"older decay condition  by the weaker condition
$1\in L^{s(\cdot)}(\R^n)$ with 
$$\frac{1}{s(x)}:=\left|\frac{1}{p(x)}-\frac{1}{p_\infty}\right|
.$$ 
\end{rem} 
\medskip

\begin{cor} \label{c1.6}
Let  $n\geq3$,  $1-n/2<\alpha<1.$ , $p(\cdot)\in \cP^{\log}(\R^n)$ and $\frac{n}{n-1+\alpha}<p_{-}\leq
p_{+}<\frac{n}{1-\alpha}.$ Then spherical maximal functions 
$\mathcal{M}^{\alpha}$ are  bounded on $L^{p(\cdot)}(\mathbb{R}^{n}).$
\end{cor}
\begin{proof}
Using Lemma~\ref{l3.3} and the complex interpolation method for variable exponent Lebesgue spaces, from Theorem~\ref{th3.2} we deduce the desired result. 
\end{proof}

\begin{thm} \label{th3.4}
Let $n\geq2$ and   $p(\cdot)\in\cP^{\log}(\R^n)$ 
with $\frac{n}{n-1+\alpha}<p_{-}\leq p_{+}
<p_{-}\frac{n-1+\alpha}{1-\alpha},\,\,0\le \alpha<1.$ Then the
spherical maximal operator $\mathcal{M}^{\alpha}$ is bounded on
$L^{p(\cdot)}(\mathbb{R}^{n}).$
\end{thm}
\begin{proof} 
Fix $\gamma$ such that 
$1<\gamma<\frac{p_{-}(n-1+\alpha)}{n}$ and $p_{+}\gamma< \frac{p_{-}(n-1+\alpha)}{1-\alpha}.$
Define the variable exponent 
$$\overline{p}(x):= \frac{p(x)n\gamma}{p_{-}(n-1+\alpha)}.$$
It is clear that $\overline{p}\in \cP^{\log}(\R^n)$,
\begin{align*}\overline{p}_{-}= \frac{n\gamma}{n-1+\alpha}> \frac{n}{n-1+\alpha}\\
\intertext{and}
\overline{p}_{+}= \frac{p_{+}n\gamma}{p_{-}(n-1+\alpha)}< \frac{n}{n-\alpha} .
\end{align*}
By Lemma~\ref{l3.3} there exists $\theta$, $0<\theta<1-\frac{2}{n}+ \frac{2}{n}\alpha$ and a variable exponent $\widetilde{p}\in \cP^{\log}(\R^n)$, such that 
$$\frac{1}{\overline{p}(\cdot)}=\frac{1-\theta}{2}+\frac{\theta}{\widetilde{p}(\cdot)}.  $$

Therefore, by the complex interpolation theorem for variable exponent Lebesgue spaces,  we have
$$L^{\overline{p}(\cdot)}(\mathbb{R}^{n}=[L^{2}(\mathbb{R}^{n},\,\,L^{\widetilde{p}(\cdot)}(\mathbb{R}^{n})]_{\theta}.$$  
From Theorem~\ref{th3.2} we deduce that the spherical maximal functions $\M^\alpha$ are  bounded on $L^{\overline{p}(\cdot)}(\mathbb{R}^{n})$.
   
By the complex interpolation method for variable exponent Lebesgue spaces, we have  
$$[L^{\infty}(\mathbb{R}^{n}),L^{\overline{p}(\cdot)}(\mathbb{R}^{n})]_{\theta}=L^{p(\cdot)}(\mathbb{R}^{n})$$
for $\theta=\frac{n\gamma}{p_{-}(n-1+\alpha)}\in (0,1)$. Finally, by using the fact that if $0\le \alpha<1$ the spherical maximal functions $\mathcal{M^{\alpha}}$ are  bounded on $L^{\infty}(\mathbb{R}^{n})$, we
obtain that the spherical maximal functions $\M^\alpha$ are  bounded on $L^{p(\cdot)}(\mathbb{R}^{n})$.
\end{proof}

Taking $\alpha=0$ in Theorem~\ref{th3.4}, we get immediately the following

\begin{cor}\label{c3.5}
Let $n\geq3$, $p(\cdot)\in \cP^{\log}(\R^n)$ and
$\frac{n}{n-1}<p_{-}\leq p_{+} <p_{-}(n-1).$ Then the spherical maximal
function $\mathcal{M}$ is bounded on $L^{p(\cdot)}(\mathbb{R}^{n}).$
\end{cor}

\medskip
\section{An application}
 
Spherical averages often make their appearance as solutions of certain 
partial differential equations. We will state two corollaries, whose proofs are immediate consequences of the boundedness of $\mathcal{M}^\alpha$ in the variable exponent Lebesgue spaces and of (\ref{sphmaxop}).

a) Let $\alpha=\frac{3-n}{2}$.  For an appropriate constant $c_n$, we have that 
$u(x,t)=c_nt\mathcal{M}^{\alpha}_{t}(x)$, where $ u$ is the solution of the wave equation
\begin{align*}
\frac{\partial^{2}u}{\partial t^{2}}(x,t)&=\Delta_{x}(u)(x,t),
\\
u(x,0)&=0, \\
\frac{\partial u}{\partial t}(x,t)&=f(x),
\end{align*}
(see \cite{S2}, page 519).

\begin{cor} 
Let  $n\geq3$ and   $p(\cdot)\in\cP^{\log}(\R^n)$.  If $\frac{2n}{n+1}<p_{-}\leq p_{+}<\frac{2n}{n-1}$,
 then for the weak solution  $u=u(x,t)$ of the wave equation
with  the initial data in  $f\in L^{p(\cdot)}(\mathbb{R}^n)$, we  have the following  a priori estimate:
\begin{equation}\label{finalest}
\left\| \sup_{t>0} \frac{| u(x, t)|}{t}\right\|_{p(\cdot)}\le C  \| f\|_{p(\cdot)},
\end{equation}
where $C$ depends on $p(\cdot)$ only.

Also, $\lim_{t\to 0}\frac{| u(x, t)|}{t}=f(x)$ a.e. $x\in \R^n$. 
\end{cor}

The estimate (\ref{finalest}) tells that the integrability property of the initial velocity of propagation of a wave is preserved in the solution, under an assumption of regularity of the exponent, expressed in terms of a logarithmic Dini-type condition. In some sense (which is made precise by such logarithmic Dini-type condition), in the places where the initial velocity does not blow up (in the sense of the boundedness of the norm in a Lebesgue space which may vary poinwise), also the wave does not. For a more detailed digression, see \cite{FGK}.

b) Let $\alpha=\frac{3-n}{2}$. For an appropriate constant $c_n$, the spherical average $u(x,t)=c_n\mathcal{M}^{\alpha}_{t}(x)$ solves  Darboux's equation (see \cite{J})
\begin{align*}
\frac{\partial^{2}u}{\partial t^{2}}(x,t)+\frac{2}{t}\frac{\partial u}{\partial t}(x,t)&=\Delta_{x}(u)(x,t),
\\
u(x,0)&=f(x), \\
\frac{\partial u}{\partial t}(x,t)&=0.
\end{align*}

\begin{cor}
Let  $n\geq3$ and   $p(\cdot)\in\cP^{\log}(\R^n)$. If $\frac{2n}{n+1}<p_{-}\leq p_{+}<\frac{2n}{n-1}$, 
 then for the weak solution $u=u(x,t)$
of the Darboux's equation  with  the initial data $f\in L^{p(\cdot)}(\mathbb{R}^n)$
 we  have the following  a priori estimate:
$$
\| \sup_{t>0}u(x, t)\|_{p(\cdot)}\le C \| f\|_{p(\cdot)},
$$
where $C$ depends on $p(\cdot)$ only.

Also, $\lim_{t\to 0}| u(x,t)|{t}=f(x)$ a.e. $x\in \R^n$. 
\end{cor}


\end{document}